\documentclass[12pt]{amsart}

\usepackage{mathrsfs}
\usepackage{amsfonts}
\usepackage{amssymb}
\usepackage[letterpaper, left=2.5cm, right=2.5cm, top=2.5cm,
bottom=2.5cm,dvips]{geometry}
\usepackage{verbatim}
\usepackage[T1]{fontenc}
\usepackage{graphicx}
\usepackage{amsmath}

\usepackage{caption}
\usepackage{hyperref}
\newtheorem{pr}{\sc Proposition}

\def\llfloor{\left\lfloor}
\def\rrfloor{\right\rfloor}

\newtheorem{theorem}{Theorem}
\theoremstyle{plain}

\newtheorem{conjecture}{Conjecture}
\newtheorem{corollary}{Corollary}

\newtheorem{lemma}{Lemma}

\newtheorem{proposition}{Proposition}

\numberwithin{equation}{section}
\usepackage{diagbox} 

\usepackage{array,multirow}

\usepackage{xcolor}

\usepackage{hyperref}
\hypersetup{
	colorlinks,
	linkcolor={blue!60!black},
	citecolor={red!60!black},
	urlcolor={red!60!black}
}

\def\er{\mathbb{R}}

\DeclareMathOperator{\vol}{vol}
\DeclareMathOperator{\sgn}{sgn}

\begin{document}
	
	\title[Neighborly boxes in $\er^d$]{New bounds on the maximum number of neighborly boxes in $\er^{d}$}
	
	\author[N. Alon]{Noga Alon}
	\address{Department of Mathematics, Princeton University, Princeton, NJ 08544, USA and Schools of
		Mathematics and Computer Science, Tel Aviv University, Tel Aviv 69978, Israel}
	\email{nalon@math.princeton.edu}
	\thanks{The first author was supported in part by NSF grant DMS-2154082 and BSF grant 2018267.}
	
	\author[J. Grytczuk]{Jaros\l aw Grytczuk}
	\address{Faculty of Mathematics and Information Science, Warsaw University
		of Technology, 00-662 Warsaw, Poland}
	\email{jaroslaw.grytczuk@pw.edu.pl}
	\thanks{The second author was supported in part by Narodowe Centrum Nauki, grant 2020/37/B/ST1/03298.}
	
	\author[A. P. Kisielewicz]{Andrzej P. Kisielewicz}
	\address{Wydzia{\l} Matematyki, Informatyki i Ekonometrii, Uniwersytet Zielonog\'orski, ul. Podg\'orna 50, 65-246 Zielona G\'ora, Poland}
	\email{A.Kisielewicz@wmie.uz.zgora.pl}
	
	\author[K. Przes\l awski]{Krzysztof Przes\l awski}
	\address{Wydzia{\l} Matematyki, Informatyki i Ekonometrii, Uniwersytet Zielonog\'orski, ul. Podg\'orna 50, 65-246 Zielona G\'ora, Poland}
	\email{K.Przeslawski@wmie.uz.zgora.pl}

		\begin{abstract}
			A family of axis-aligned boxes in $\er^d$ is \emph{$k$-neighborly} if the intersection of every two of them has dimension at least $d-k$ and at most $d-1$. Let $n(k,d)$ denote the maximum size of such a family. It is known that $n(k,d)$ can be equivalently defined as the maximum number of vertices in a complete graph whose edges can be covered by $d$ complete bipartite graphs, with each edge covered at most $k$ times.
			
			We derive a new upper bound on $n(k,d)$, which implies, in particular, that $n(k,d)\leqslant (2-\delta)^d$ if $k\leqslant (1-\varepsilon)d$, where $\delta>0$ depends on arbitrarily chosen $\varepsilon>0$. The proof applies a classical result of Kleitman, concerning the maximum size of sets with a given diameter in discrete hypercubes. By an explicit construction we obtain also a new lower bound for $n(k,d)$, which implies that $n(k,d)\geqslant (1-o(1))\frac{d^k}{k!}$. We also study $k$-neighborly families of boxes with additional structural properties. Families called \emph{total laminations}, that split in a tree-like fashion, turn out to be particularly useful for explicit constructions. We pose a few conjectures based on these constructions and some computational experiments.
		\end{abstract}
		
		\maketitle
		
		\section{Introduction}
		
		Let $k,d$ be two positive integers, with $k\leqslant d$. A {\it standard box} in $\er^d$ is a set of the form $K=K_1\times \cdots \times K_d$, where $K_i\subset \er$ is a closed interval, for each $i\in [d]$, where $[d]=\{1,2,\ldots, d\}$. Two standard boxes $K,L\subset \er^d$ are {\it k-neighborly} if $d-k\leqslant {\rm dim}(K\cap L)\leqslant d-1$. A family $\mathscr{F}$ of standard boxes is {\it k-neighborly} if every two boxes in $\mathscr{F}$ are $k$-neighborly. 
		
		Let $n(k,d)$ be the maximum possible cardinality of a family $\mathscr{F}$ of $k$-neighborly standard boxes in $\er^d$. In 1985 Zaks \cite{Zaks3} proved that $n(1,d)=d+1$ by relating the problem to the well-known theorem of Graham and Pollak \cite{GP} on bipartite decompositions of complete graphs. It is also not hard to demonstrate that $n(d,d)=2^d$. More generally, as explained in \cite{Al}, $n(k,d)$ is the maximum possible number of vertices in a complete graph that can be covered by $d$ complete bipartite subgraphs (\emph{bicliques}) so that every edge is covered at least once and at most $k$ times. Using this interpretation, the following general bounds were proved in \cite{Al}:
		\begin{equation}\label{Eq Alon}
			\left(\frac{d}{k}\right)^k\leqslant \prod_{i=0}^{k-1}\left(\left\lfloor \frac{d+i}{k}\right\rfloor+1\right) \leqslant n(k,d) \leqslant \sum_{i=0}^k2^i\binom{d}{i}< 2\cdot (2e)^k\cdot \left(\frac{d}{k}\right)^k.
		\end{equation}
		An improvement of the upper bound, namely, $n(k,d)\leqslant 1+ \sum_{i=1}^k2^{i-1}\binom{d}{i}$, was obtained by Huang and Sudakov in \cite{HS} (see also \cite{CT} for a generalization). These upper bounds are not very far from the lower bound for $k$ much smaller than $d$. For large $k$, however, the gap is very large. In particular, for $k >0.23 d$ and large $d$ these bounds are larger than the trivial $2^d$ upper bound which holds for every $k \leqslant d$.
		
		In the present paper we give a new upper bound on $n(k,d)$, which improves the above estimate if $k$ is close enough to $d$. It is stated in the following theorem, where an easy lower bound on $n(k,d)$ is also included, for comparison.
		
		\begin{theorem}\label{no} For every $1\leqslant k \leqslant d-1$,
			\begin{equation}\label{Main Theorem eq1}
				n(k,d) \geqslant \sum_{i=0}^{\lfloor k/2 \rfloor} {d \choose i}.
			\end{equation}
			For every $1 \leqslant k \leqslant k +2t-2 \leqslant d-1$, where $t$ is an arbitrary
			positive integer,
			\begin{equation}\label{Main Theorem eq2}
				n(k,d) \leqslant 2^{d-t}+\sum_{i=0}^{\lceil (k+2t-2)/2 \rceil}  {d \choose i}.
			\end{equation}
		\end{theorem}
		
		Two simple corollaries of the result above are the following.
		\begin{corollary}\label{Cor1}
			\label{c12}
			For any $\varepsilon >0$ there is
			$\delta>0$ so that for any $k \leqslant (1-\varepsilon)d$,
			\begin{equation}
				n(k,d) \leqslant (2-\delta)^d.
			\end{equation}
			
		\end{corollary}
		\begin{corollary}\label{Cor2}
			\label{c13}
			If $d,k$ tend to infinity
			so that
			$d-k \mapsto \infty$ and $d-k=o(d^{2/3})$, then
			\begin{equation}
				n(k,d) =(1+o(1)) \sum_{i=0}^{\lfloor k/2 \rfloor} {d \choose i}.
			\end{equation}
		\end{corollary}
		Note that a special case of Corollary \ref{c13} gives that
		for any $k=d-s$ with $s$ tending to infinity and satisfying 
		$s=o(d^{1/2})$,
		$n(k,d)=(1+o(1))2^{d-1}.$
		
		The proof of Theorem \ref{no} is quite simple, however, it relies on a fundamental result of Kleitman \cite{K} (conjectured by Erd\H{o}s) concerning the maximum size of sets with given diameter in a discrete hypercube (see Section 2). Actually, by a more detailed analysis one may derive slightly more precise upper bounds on $n(k,d)$ (see Subsection 2.7).
		
		In Section 3 we present further results on neighborly families of boxes and some consequences for the related problem of biclique coverings of complete graphs. We will use there some observations stemmed from the proof of the following fact, giving another precise value of the function $n(k,d)$.
		
		\begin{proposition}\label{Proposition n(d-1,d)}
			For every $d\geqslant 2$ we have
			\begin{equation}
				n(d-1,d)=3\cdot 2^{d-2}.
			\end{equation}
		\end{proposition}
		
		Let us remark that following the application of Kleitman's theorem in the study of this problem described in the first arXiv version of the present paper, Cheng, Wang, Xu, and Yip \cite{ChWX} also obtained new upper bounds on $n(k,d)$. For some values of parameters $k$ and $d$ their bounds are better than ours. Interestingly, for small values of $d$ and $k$, they determined exact values of $n(k,d)$ (see Table \ref{table:Gurobi} and tables in \cite{ChWX}).
		
		We also provide a general construction of $k$-neighborly families of boxes which gives a substantial improvement of the existing lower bounds on $n(k,d)$, for every fixed $k\geqslant 2$ and sufficiently large $d$.
		
		\begin{theorem}\label{Theorem Lower Bound}
			Let $2\leqslant k \leqslant m \leqslant d$ be given integers such that $d\geqslant \binom{m}{k}-1+m$. Then
			\begin{equation}
				n(k,d)\geqslant \left(d-\binom{m}{k}+1\right)^k\cdot \frac{\binom{m}{k}}{m^k}.
			\end{equation}
			In consequence,
			\begin{equation}
				n(k,d)\geqslant (1-o(1))\frac{d^k}{k!}.
			\end{equation}
		\end{theorem}
		
		For instance, for $k=2$ we get $n(2,d)\geqslant (1-o(1))\frac{d^2}{2}$, which improves the previous lower bound, $n(2,d)\geqslant \frac{d^2}{4}$. As mentioned earlier by Alon \cite{Al} and Huang and Sudakov \cite{HS}, determining the precise asymptotic order of $n(2,d)$ seems to be a hard task. In the final section we pose a general conjecture which, together with the above lower bound would imply that $$\lim_{d\rightarrow \infty}\frac{n(2,d)}{d^2}=\frac{1}{2}.$$
		However, as yet we do not even know whether the limit above exists.
		
		\section{Proofs of the main results}
		
		\subsection{Sets of maximum size and given diameter in the Hamming cube}
		Let $H^d=\{0,1\}^d$ denote the \emph{Hamming cube} of dimension $d$, that is, the set of binary \emph{strings} of length $d$ representing the vertices of the cube, with the \emph{Hamming distance} between two vertices defined by $h(x,y)=|\{i\in [d]\colon x_i\neq y_i\}|$. Two strings $x,y\in H^d$ are called \emph{complementary} if $x_i\neq y_i$ for all $i\in [d]$. For a subset $A\subseteq H^d$, let $D(A)=\max\{h(x,y)\colon x,y\in A\}$ denote the \emph{diameter} of $A$. By $B_t(x)\subseteq H^d$ we denote the \emph{ball} of radius $t$ centered at $x\in H^d$. Notice that the size of $B_t(x)$ is equal to $\sum_{i=0}^{t}\binom{d}{i}$, while the diameter $D(B_t(x))$ equals $2t$, for $t<d/2$.
		
		In 1961 Erd\H{o}s conjectured that no set of diameter $2t\leqslant d-1$ in $H^d$ can have more points than the ball $B_t(x)$. This was confirmed by Kleitman \cite{K} in 1966. As observed by Bezrukov \cite{B}, a similar assertion holds for sets of odd diameters, with an extremal example being formed of two balls $B_t(x)\cup B_t(y)$ centered at any two points $x,y\in H^d$, with $h(x,y)=1$. Bezrukov \cite{B} also proved that there are no other extremal sets as long as $k<d-1$. If $k=d-1$, then, by elementary reasoning, every extremal set consists of exactly one element from every complementary pair.
		
		\begin{theorem}[Kleitman \cite{K}, Bezrukov \cite{B}]\label{Theorem Kleitman}
			Let $A\subseteq H^d$ be a set of diameter $k\leqslant d-1$ and maximum size. If $k=2t$, then
			\begin{equation}\label{Eq Kleitman 1}
				|A|\leqslant \sum_{i=0}^t\binom{d}{i}.
			\end{equation}
			If $k=2t+1$, then
			\begin{equation}\label{Eq Kleitman 2}
				|A|\leqslant \binom{d-1}{t}+ \sum_{i=0}^t\binom{d}{i}.
			\end{equation}
		\end{theorem}
		
		\subsection{Neighborly boxes, bipartite coverings, and squashed cubes}
		
		Let $S=\{0,1,\ast\}$ be a set of symbols and let $S^d$ be the set of all strings of length $d$ over $S$. The elements of $S^d$ can be interpreted as vertices of the Hamming cube $H^d$ with some positions replaced by the new symbol $\ast$, called \emph{joker}. In this way, any string $x \in S^d$ defines a subcube $H(x)$ of $H^d$, where $H(x)$ consists of all strings obtained by changing all the jokers in $x$ to $0$ or $1$ in all possible ways. The number of jokers in $x$, denoted as $j(x)$ is the dimension of the subcube $H(x)$, and its cardinality is $2^{j(x)}$.
		
		For, $x,y\in S^d$, let $d(x,y)$ denote the number of positions in which one string has $0$ and the other has $1$. Note that if $d(x,y) \geqslant 1$, then the subcubes $H(x)$ and $H(y)$ are pairwise disjoint. Note also that if the dimension of $H(x)$ is $i$ and that of $H(y)$ is $j$, then any two binary strings, $u \in H(x)$ and $v \in H(y)$, differ in at most $d(x,y)+i+j$ coordinates. 
		
		Another interpretation of strings in $S^d$ is connected to a special family $\mathscr{S}^d$ of standard boxes formed of just three different intervals, $[-1,0]$, $[0,1]$, and $[-1,1]$. Given such a box $K=K_1\times \cdots \times K_d$, with $K_i$ being one of those three intervals, we may assign to it a unique string $v(K)=v_1\cdots v_d$ in $S^d$ by the formula:
		$$v_i=
		\begin{cases}
			0, & \text{if $K_i=[-1,0]$},\\
			1, & \text{if $K_i=[0,1]$},\\
			\ast, & \text{if $K_i=[-1,1]$.}
		\end{cases}
		$$ 
		
		It is not hard to verify that two boxes $K,L\in \mathscr{S}^d$ are $k$-neighborly if and only if the corresponding strings, $v(K)$ and $v(L)$, differ on at least one, and on at most $k$ non-joker positions (in both strings). In other words, if we add up the two strings $v(K)$ and $v(L)$ as vectors, with additional rule that $i+\ast=\ast$, for each $i\in S$, then the resulting string $v(K)+v(L)$ should have its number of $1$'s in the interval $[1,k]$.
		
		Moreover, it can be easily demonstrated that for every family $\mathscr{F}$ of $k$-neighborly standard boxes in $\er ^d$ there is an equivalent family of boxes in $\mathscr{S}^d$ (preserving dimensions of all intersections of members of the family $\mathscr{F}$). A detailed argument can be found in \cite{Al}. Hence, the problem of determining the function $n(k,d)$ can be reduced to families of boxes in $\mathscr{S}^d$ and studied through their representing strings in $S^d$.
		
		A connection between $S^d$ and bipartite coverings is obtained by assigning to each vertex $v$ of a complete graph covered by $d$ bicliques with vertex classes $(L_i,R_i)$, $1 \leqslant i \leqslant d$, a string in $S^d$, whose $i$-th coordinate is $0$ if $v \in L_i$, it is $1$ if $v \in R_i$ and it is $*$ if $v \not \in L_i \cup R_i$. Therefore, $n(k,d)$ is the maximum possible number of strings in a family $F \subseteq S^d$ so that $1 \leqslant d(x,y) \leqslant k$ holds for any two distinct $x,y \in F$.
		
		\subsection{Proof of Proposition \ref{Proposition n(d-1,d)}}
		Let $\mathscr{F}$ be a family of $k$-neighborly boxes in $\mathscr{S}^d$, or, equivalently, a family $F$ of strings in $S^d$ such that $1\leqslant d(x,y)\leqslant k$, for each pair $x,y\in F$. Assume that the size of $F$ is as large as possible. Denote by $F^{\ast}$ the subset of $F$ consisting of all strings with at least one joker, and set $F_0=F\setminus F^{\ast}$.
		
		As explained above, every string $x\in F$ corresponds to a subcube $H(x)$ of $H^d$ whose dimension is exactly the number of jokers in $x$. Since every pair of strings in $F$ differs on at least one non-joker position, these subcubes are pairwise disjoint. Hence, the number of strings in $F^{\ast}$ cannot be greater than half of the total number of vertices spanned by the corresponding subcubes. Consequently,
		\begin{equation}
			|F^{\ast}|\leqslant \frac{1}{2}\cdot|H^d\setminus F_0|.
		\end{equation}
		So, we may write
		\begin{equation}
			|F|=|F_0|+|F^{\ast}|\leqslant |F_0|+ \frac{1}{2}\cdot(2^d- |F_0|)=2^{d-1}+\frac{1}{2}\cdot|F_0|.
		\end{equation}
		
		Suppose now that $k=d-1$ and $F$ is a $k$-neighborly family. Then $F_0$ cannot contain any complementary pair of strings from $H^d$, and therefore $|F_0|\leqslant \frac{1}{2}\cdot |H^d|=2^{d-1}$. In this way we get
		\begin{equation}
			|F|\leqslant 2^{d-1}+2^{d-2}=3\cdot 2^{d-2}.
		\end{equation}
		On the other hand, one may easily produce a family $F$ matching this upper bound. For instance, one may take $F_0=\{v\in H^d: v_1=0\}$ and $F^*=\{x\in S^d:x_1=1, x_2=\ast, x_3\cdots x_d\in H^{d-2}\}$, which completes the proof. \hfill{$\square$}
		
		\subsection{Proof of Theorem \ref{no}}
		The proof of inequality (\ref{Main Theorem eq1}) is very simple. Indeed, the family $F$ of all binary strings of length $d$ in a Hamming ball of radius $\lfloor k/2 \rfloor$ is a collection of $\sum_{i=0}^{\lfloor k/2 \rfloor} {d \choose i}$ binary strings. These can be viewed as strings in $S^d$ (containing no jokers). For any two distinct $x,y \in F$, we have $1 \leqslant d(x,y) \leqslant 2\lfloor k/2 \rfloor \leqslant k$, which implies (\ref{Main Theorem eq1}).
		
		The proof of inequality (\ref{Main Theorem eq2}) is also short, by using Theorem \ref{Theorem Kleitman}. Let $F \subseteq S^d$ be a family of strings so that for any two distinct $x,y \in F$, we have
		$1 \leqslant d(x,y) \leqslant k$. For each $i$, $1 \leqslant i \leqslant d$, let $F_i$ be the subset of $F$ consisting of all strings in $F$ whose number of jokers is exactly $i$. Let $f_i=|F_i|$ be the cardinality of $F_i$.
		
		Since every subcube $H(x)$, for $x \in F_i$, contains $2^i$ points, and all these subcubes
		are pairwise disjoint, it follows that
		\begin{equation}
			\label{e21}
			\sum_{i=0}^d 2^i f_i \leqslant 2^d~~\mbox{and therefore, for every}~~
			t, ~~ \sum_{i=t}^d f_i \leqslant 2^{d-t}.
		\end{equation}
		Fix a  positive integer $t$ so that $k+2t-2<d$. For each $i$, $0 \leqslant i \leqslant t-1$, and for each $x \in F_i$, let $x'$ be an arbitrary binary string in $H(x)$. (For example, one can define $x'$ to be the string obtained from $x$ by replacing each of its jokers by $0$.) Then all these strings $x'$ corresponding to members $x \in	\bigcup_{i=0}^{t-1} F_i$ are distinct, and every two of them
		differ in at most $k+2t-2$ positions. It thus follows from Theorem \ref{Theorem Kleitman} that
		\begin{equation}
			\label{e22}
			f_0+f_1+ \ldots +f_{t-1} \leqslant \sum_{i=0}^{\lceil (k+2t-2)/2 \rceil}  {d \choose i}.
		\end{equation}
		By (\ref{e21}) and (\ref{e22}), we obtain
		$$
		|F| = \sum_{i=0}^d f_i =(f_0+f_1+ \ldots f_{t-1})
		+\sum_{i=t}^d f_i \leqslant \sum_{i=0}^{\lceil (k+2t-2)/2 \rceil}  {d \choose i}
		+2^{d-t},
		$$
		which gives the inequality (\ref{Main Theorem eq2}). The proof is complete. \hfill{$\square$}
		
		\subsection{Proof of Corollary \ref{Cor1}}In (\ref{Main Theorem eq2}) of Theorem \ref{no}
		choose $t=\frac{\varepsilon d}{4}$ and apply the
		standard estimates for binomial distributions (cf., e.g.,
		\cite{AS}, Theorem A.1.13) to conclude 
		that the assertion of the corollary holds with
		$\delta = \Omega(\varepsilon^2)$.
		
		\subsection{Proof of Corollary \ref{Cor2}}
		Put $k=d-s$. By assumption $d$ and $s$ tend to infinity,
		and $s=o(d^{2/3})$. If $s \leqslant 10 \sqrt d$, apply
		inequality (\ref{Main Theorem eq2}) of Theorem \ref{no}, with $t=2 \log s$ (in fact
		any $t$ tending to infinity and satisfying $t=o(s)$ will do)
		to obtain the desired result.
		
		If $s \geqslant 10 \sqrt d$ (and $s=o(d^{2/3})$), apply (\ref{Main Theorem eq2}) of Theorem \ref{no} with $t=10s^2/d$. By assumption $t=o(d/s)$ and it is not difficult to verify that this implies that
		$$
		\sum_{i=0}^{\lceil (k+2t-2)/2 \rceil}  {d \choose i}
		=(1+o(1)) \cdot \sum_{i=0}^{\lfloor k/2 \rfloor} {d \choose i}.
		$$
		It is also not too difficult to check that with this choice of $t$,
		$$
		2^{d-t} =o\left(\sum_{i=0}^{\lfloor k/2 \rfloor} {d \choose i}\right).
		$$
		The desired result follows by plugging these two inequalities in (\ref{Main Theorem eq2}) of Theorem \ref{no}.
		\hfill $\Box$
		
		\subsection{An estimation of $n(k,d)$ by a combination of sizes of maximal sets}
		
		Suppose that $F$ is a family of $k$-neighborly subcubes of $H^d$. Again, let $F_i$ be the subfamily of $F$ that consists of all subcubes of dimension $i$ or less. Let $f_j$ be the number of all subcubes of dimension $j$ belonging to $F$. Then the number of elements of $H^d $ covered by $F_i$ is equal $f_0+2f_1+\cdots + 2^if_i$. Moreover, the Hamming distance between any two such elements does not exceed $k+2i$.
		
		Let $\kappa (s, d)$ denotes the maximum size of sets whose diameter is not greater than $s$. In particular, $\kappa(s,d)=2^d$, for $s\geqslant d$. Then 
		\[
		f_0+2f_1+\cdots +2^if_i\leqslant \kappa(k+2i,d).  \eqno{(\kappa_i)}
		\]
		A trivial constraint
		\[
		f_0+2f_1+\cdots + 2^{d-1} f_{d-1}\leqslant 2^d. \eqno{(\tau)}
		\]
		is identical with $(\kappa_i)$ for $i=d-1$ and $k\geqslant 2$.
		
		We want to bound from above $|F|=f_0+f_1+\ldots+f_{d-1}$. It can be achieved by considering an integer optimization problem where the objective is to maximize $f_0+f_1+\ldots+f_{d-1}$ subject to the constraints given by 
		$(\kappa_i)$. Among all the optimal solutions, let us choose the one, say $f_0=a_0, \ldots, f_{d-1}=a_{d-1}$,  which is maximal with respect to the lexicographical order. We claim that it satisfies the following inequalities:
		\[
		\kappa(k+2i,d) < f_0+2f_1+\cdots +2^if_i +2^i,  \eqno{(\kappa^+_i)}.
		\]
		
		If not, let us choose $i=s$ for which ($\kappa^+_i$) does not hold. Then, we can define a new sequence $(b_i\colon i=1,\ldots, d-1)$ as follows.
		\begin{itemize}
			\item Let us set  $b_s =a_s+1$, and let $t>s$ be the first index, if there is any, for which $a_t>0$.
			\item Let us set $b_t=a_t-1$, and for all the remaining indices $u$, let $b_u=a_u$. 
		\end{itemize}
		It is rather clear that since $(\kappa(k+2i,d)\colon i\leqslant d-1)$ is increasing and $2^s< 2^t$, the resulting sequence satisfies all the constraints. Moreover, it also has to be optimal. On the other hand, it is larger with respect to the lexicographical order, which is a contradiction.
		
		Observe now that if $d-k$ is divisible by $2$, then plugging $i=\frac{d-k}2$ into $(\kappa^+_i)$ gives us
		$$
		2^d  <  f_0+2f_1+\cdots +2^{\frac{d-k}2} f_{\frac{d-k}2} +2^{\frac{d-k}2}, 
		$$
		which readily implies, that $a_j=0$ for $j>{\frac{d-k}2}$.  If $d-k$ is not divisible by $2$, then plugging $i=\llfloor\frac{d-k}2\rrfloor$ into $(\kappa^+_i)$  and $(\kappa_i)$ yields
		$$
		f_0+2f_1+\cdots +2^{\llfloor\frac{d-k}2\rrfloor} f_{\llfloor\frac{d-k}2\rrfloor}\leqslant\kappa(d-1,d) = 2^{d-1}  <  f_0+2f_1+\cdots +2^{\llfloor\frac{d-k}2\rrfloor} f_{\llfloor\frac{d-k}2\rrfloor} +2^{\llfloor\frac{d-k}2\rrfloor}. 
		$$
		Since the sequence $f_i=a_i$  is maximal with respect to the lexicographical order, and satisfies the constraints including ($\tau$), we deduce that 
		$$
		a_{\llfloor\frac{d-k}2\rrfloor+1}=2^{d -\llfloor\frac{d-k}2\rrfloor-2}, \eqno{(\lambda)} 
		$$
		and $a_j=0$  for $j > \llfloor\frac{d-k}2\rrfloor+1$.
		
		Now we are ready to estimate the optimal value of the objective function. Suppose that $0\le i-1< i\leqslant \frac{d-k}2$.  Then, by ($\kappa^+_{i-1}$) and ($\kappa_i$), we obtain
		$$
		\kappa(k+2i,d) -\kappa(k+2(i-1),d)\geqslant 2^i a_i - 2^{i-1}.
		$$
		Consequently,
		$$
		a_i < \frac 1 {2^i}\kappa(k+2i,d) -\frac 1{2^i}\kappa(k+2(i-1),d) +\frac 1 2
		$$ 
		For $x \in \mathbb R$, let us denote $\lfloor x \rceil=\lceil x- 1\rceil$. Since $a_i$ is an integer, we have  
		$$
		a_i\leqslant \left\lfloor \frac 1 {2^i}\kappa(k+2i,d) -\frac 1{2^i}\kappa(k+2(i-1),d) +\frac 1 2\right\rceil.
		$$
		Suppose first that $k$ is even. Since, by Theorem \ref{Theorem Kleitman}, $\kappa(s,d)=\sum_{j=0} ^{\frac{s}{2}} \binom{d}{j}$, whenever $s < d$ and $s$ is even, we have 
		$$
		\kappa(k+2i,d) -\kappa(k+2(i-1),d) =\binom{d}{\frac{k}2+i}.
		$$  
		Therefore, 
		$$
		a_i\leqslant \left\lfloor \frac 1 {2^i}\binom{d}{\frac k 2 +i} +\frac 1 2\right\rceil \leqslant \frac 1 {2^i}\binom{d}{\frac k2+i} +\frac 1 2, \eqno{(\eta)}
		$$
		whenever $1 \leqslant i < \frac{d-k} 2$. If  $i=\frac{d-k} 2$, that is, also $d$ is even, then
		$$
		\kappa(k+2i,d) -\kappa(k+2(i-1),d) = \kappa(d,d) -\kappa(d-2,d) =2^d-\sum_{s=0}^{\frac{d-2}2}\binom{d}s=
		2^{d-1}+\frac 1 2 \binom{d}{\frac d 2}.
		$$  
		Thus,  
		$$
		a_{\frac{d-k}{2}}\leqslant \left\lfloor 2^{\frac{d+k} 2-1} + \frac 1 {2^{\frac{d-k}2 +1}} \binom{d}{\frac{d}2} +\frac 1 2\right\rceil.\eqno{(\eta')}
		$$
		Suppose now that  $k$ is odd. Since, by Theorem \ref{Theorem Kleitman}, $\kappa(s,d)=\binom{d-1} {\llfloor\frac s2\rrfloor}+ \sum_{j=0} ^ {\llfloor\frac s2\rrfloor} \binom{d}{j}$, whenever $s < d$ and $s$ is odd, we have
		\begin{eqnarray*}
			\kappa(k+2i,d) -\kappa(k+2(i-1),d)  &= & \binom{d-1}{\llfloor\frac k2\rrfloor+i} - \binom{d-1}{\llfloor\frac k2\rrfloor+i-1} + \binom{d}{\llfloor\frac k2\rrfloor+i}\\
			&=& 2\binom{d-1}{\llfloor\frac k2\rrfloor+i},
		\end{eqnarray*}
		whenever $1\leqslant i <\frac{d-k}2$.
		Therefore, 
		$$
		a_i\leqslant \left\lfloor \frac 1 {2^{i-1}}\binom{d-1}{\llfloor\frac k2\rrfloor+i} +\frac 1 2 \right\rceil \leqslant \frac 1 {2^{i-1}}\binom{d-1}{\llfloor\frac k2\rrfloor+i} +\frac 1 2,  \eqno{(\omega)}
		$$
		whenever $1 \leqslant i < \frac{d-k} 2$. If $i=\frac{d-k} 2$, that is, also $d$ is odd 
		$$
		\kappa(k+2i,d) -\kappa(k+2(i-1),d)  
		=  2^d- \binom{d-1}{\llfloor\frac{d-2}2\rrfloor}-\sum_{s=0}^{\llfloor\frac{d-2}2\rrfloor} \binom{d}{s} 
		= 2^{d-1}+\binom{d-1}{\frac {d-1} 2}.
		$$
		Thus,
		$$
		a_{\frac{d-k}{2}}\leqslant \left\lfloor 2^{\frac{d+k} 2-1} + \frac 1 {2^{\frac{d-k}2}} \binom{d-1}{\frac{d-1}2} +\frac 1 2\right\rceil. \eqno{(\omega')}
		$$
		
		\textbf{If $k$ and $d$ are even}, then summing up $a_i$,  with respect to $i=1,\ldots,  \frac {d-k} 2$ and bearing in mind that $a_0\leqslant \kappa(k,d)$ by $(\eta)$ and $(\eta')$ we obtain
		\begin{equation}
			\label{2even}
			\begin{split}
				n(k,d) \leqslant& \sum_i a_i \le \sum_{s=0}^{\frac{k}2} \binom{d}{s}+  \sum_{i=1}^{\frac{d-k}2-1} \left\lfloor \frac 1 {2^i}\binom{d}{\frac k2+i} +\frac 1 2\right\rceil   + \left\lfloor 2^{\frac{d+k} 2-1} + \frac 1 {2^{\frac{d-k}2 +1}} \binom{d}{\frac{d}2} +\frac 1 2\right\rceil \\ 
				\leqslant&  \frac{d-k}4 +2^{\frac{d+k} 2-1} + \frac 1 {2^{\frac{d-k}2 +1}} \binom{d}{\frac{d}2}+ \sum_{s=0}^{\frac{k}2} \binom{d}{s}+  \sum_{i=1}^{\frac{d-k}2 -1} \frac 1 {2^i}\binom{d}{\frac k2+i}. 
			\end{split}
		\end{equation}
		
		\textbf{If $k$ is even and $d$ is odd}, then taking into account $(\lambda)$ we get 
		\begin{equation}
			\begin{split}
				n(k,d) 
				\leqslant &  2^{d -\llfloor\frac{d-k}2\rrfloor-2} + \sum_{s=0}^{\frac{k}2} \binom{d}{s}+  \sum_{i=1}^{\llfloor\frac{d-k}2\rrfloor} \left\lfloor\frac 1 {2^i}\binom{d}{\frac k2+i} +\frac 1 2\right\rceil \\
				\leqslant &  \frac 1 2\llfloor\frac{d-k}2\rrfloor +2^{d -\llfloor\frac{d-k}2\rrfloor-2} +\sum_{s=0}^{\frac{k}2} \binom{d}{s}+  \sum_{i=1}^{\llfloor\frac{d-k}2\rrfloor} \frac 1 {2^i}\binom{d}{\frac k2+i}.
			\end{split}
		\end{equation}
		
		\textbf{If $k$ is odd and $d$ is even}, then by $(\omega)$, $(\lambda)$ and the expression on $\kappa(k,d)$ we get
		\begin{equation}
			\begin{split}
				n(k,d) 
				\leqslant &  2^{d -\llfloor\frac{d-k}2\rrfloor-2} + \binom{d-1} {\llfloor\frac k2\rrfloor}+ \sum_{s=0} ^ {\llfloor\frac k 2\rrfloor} \binom{d}{s}+  
				\sum_{i=1}^{\llfloor\frac{d-k}2\rrfloor} \left\lfloor \frac 1 {2^{i-1}}\binom{d-1}{\llfloor\frac k2\rrfloor+i} +\frac 1 2\right\rceil \\
				\leqslant &  \frac 1 2\llfloor\frac{d-k}2\rrfloor +2^{d -\llfloor\frac{d-k}2\rrfloor-2} + \binom{d-1} {\llfloor\frac k2\rrfloor}+\sum_{s=0}^{\llfloor\frac{k}2\rrfloor} \binom{d}{s}+  \sum_{i=1}^{\llfloor\frac{d-k}2\rrfloor} \frac 1 {2^{i-1}}\binom{d-1}{\llfloor\frac k 2\rrfloor+i}.
			\end{split}
		\end{equation}
		
		\textbf{If both $k$ and $d$ are odd}, then we proceed similarly to the case $k$ and $d$ are even, however, we have to apply $(\omega)$ and $(\omega')$ instead of $(\eta)$  and $(\eta')$:
		\begin{equation}
			\begin{split}
				n(k,d) 
				\leqslant & \binom{d-1} {\llfloor\frac k2\rrfloor}+ \sum_{s=0} ^ {\llfloor\frac k 2\rrfloor} \binom{d}{s}+  
				\sum_{i=1}^{\frac{d-k}2-1} \left\lfloor \frac 1 {2^{i-1}}\binom{d-1}{\llfloor\frac k2\rrfloor+i} +\frac 1 2\right\rceil  + \left\lfloor 2^{\frac{d+k} 2-1} + \frac 1 {2^{\frac{d-k}2}} \binom{d-1}{\frac{d-1}2} +\frac 1 2\right\rceil \\
				\leqslant &  \frac{d-k}4  + \binom{d-1} {\llfloor\frac k2\rrfloor}+2^{\frac{d+k} 2-1} + \frac 1 {2^{\frac{d-k}2}} \binom{d-1}{\frac{d-1}2}+\sum_{s=0}^{\llfloor\frac{k}2\rrfloor} \binom{d}{s}+  \sum_{i=1}^{\frac{d-k}2} \frac 1 {2^{i-1}}\binom{d-1}{\llfloor\frac k 2\rrfloor+i}.
			\end{split}
		\end{equation}
		
		\subsection{Another way of getting the lower bound in (\ref{Eq Alon})}
		
		We will prove a general result, which immediately implies the lower bound for $n(k,d)$ given in \cite{Al}.
		\begin{pr}\label{Proposition Lower Bound}
			Let $1\leqslant k_i\leqslant d_i$, $i=1,\ldots, s$, be given integers. Then
			\begin{equation}
				n(k_1,d_1)n(k_2,d_2)\cdots n(k_s,d_s)\leqslant n(k_1+\cdots +k_s, d_1+\cdots +d_s). 
			\end{equation} 
		\end{pr}
		\begin{proof}
			For two families $F\subseteq S^d$ and $G\subseteq S^m$, let $FG=\{uv\colon u\in F,v\in G\}$ be the family of strings in $S^{d+m}$ consisting of all possible \emph{concatenations} of strings from $F$ and $G$. Clearly, we have $|FG|=|F|\cdot |G|$. In case of singleton families $\{v\}$ we will write simply $vF$ instead of $\{v\}F$, and also $v^m$ for a concatenation of $m$ copies of a string $v$. 
			
			Let $F_i$ denote a $k_i$-neighborly family of maximum size in $S^{d_i}$, for $i=1,\ldots, s$. It is not hard to see that $F=F_1F_2 \cdots F_s$ is a $k$-neighborly family in $S^d$, where $k=k_1+\cdots +k_s$ and $d=d_1+\cdots +d_s$, which proves the assertion.
		\end{proof}
		
		The lower bound for $n(k,d)$ in (\ref{Eq Alon}) may be derived now easily by applying the trivial inequality $n(1,d)\geqslant d+1$ (which actually is an equality by the Graham-Pollak Theorem). Indeed, put in the above result $s=k$, $k_i=1$, and $d=d_1+\cdots +d_k$. Then
		\begin{equation}
			n(k,d)\geqslant n(1,d_1)\cdots n(1,d_k)=(d_1+1)\cdots (d_k+1).
		\end{equation}
		The last product is maximized when $d$ is partitioned most evenly into summands $d_i$, which coincides with the product $\prod_{i=0}^{k-1}\left(\left\lfloor \frac{d+i}{k}\right\rfloor+1\right)$.
		
		\subsection{Proof of Theorem \ref{Theorem Lower Bound}}
		
		Let $k\leqslant m\leqslant d$ be fixed. Suppose that $d=a_1+\cdots +a_m+\binom{m}{k}-1$, where $a_i$ are positive integers. Let $\binom{[m]}{k}$ denote the set of all $k$-element subsets of $[m]$. For a subset $B\in \binom{[m]}{k}$, denote $p_B=\prod_{i\in B} (a_i+1)$. We will describe an explicit construction of a $k$-neighborly family of strings in $S^d$ with exactly $\sum_{B\in \binom{[m]}{k}}p_B$ members.
		
		Let $A_i\subseteq S^{a_i}$ be any $1$-neighborly family of size $a_i+1$, for $i=1,\ldots, m$. For a fixed subset $B\in \binom{[m]}{k}$, consider a set $R_B$ of strings in $S^{a_1+\cdots+a_m}$ defined by $R_B=X_1\cdots X_m$, where $X_i=A_i$ if $i\in B$ and $X_i=\{*^{a_i}\}$, otherwise. This construction can be described in words as follows. 
		Split the set $[a_1+\cdots +a_m]$ into consecutive intervals $I_1,\ldots, I_m$ of lengths $a_1,\ldots, a_m$, respectively. Next, for each $i\in B$, fill the interval $I_i$ with any string from $A_i$, while every other interval $I_j$, fill up with jokers. Clearly, we have $|R_B|=p_B$. Moreover, denoting $R=\bigcup_{B\in \binom{[m]}{k}}R_B$, we have $|R|=\sum_{B\in \binom{[m]}{k}}p_B$, since no two constructed strings are the same. 
		
		It is not hard to see that any two strings $u,v\in R_B$ satisfy $1\leqslant d(u,v)\leqslant k$. Also, it is not hard to check that for any two different subsets $B'$ and $B''$ in $\binom{[m]}{k}$, and any strings $u\in B'$ and $v\in B''$, we must have $d(u,v)\leqslant k-1$, though sometimes we may have $d(u,v)=0$. This last obstacle may be easily fixed as follows.
		
		Let $A$ be any $1$-neighborly family in $S^{\binom{m}{k}-1}$ of size $\binom{m}{k}$. So, the elements of $A$ may be indexed by subsets $B$ in $\binom{[m]}{k}$  as $v_B$. For each $B\in \binom{[m]}{k}$, we may now form a new family $R'_B=v_BR_B=\{v_Bx: x\in R_B\}$ by appending the string $v_B$ at the beginning of every string $x$ in $R_B$. Now, every pair of strings $u,v\in R'=\bigcup_{B\in \binom{[m]}{k}}R'_B$ satisfies $1\leqslant d(u,v)\leqslant k$, which means that $R'$ is $k$-neighborly. This completes construction of a $k$-neighborly family with the aforementioned size, as $|R'|=|R|$.
		
		It remains to demonstrate that choosing appropriate numbers $a_1,\ldots ,a_m$, summing up to $t=d-\binom{m}{k}+1$, gives the asserted lower bound. It is not hard to see that the product $p_B=\prod_{i\in B}(a_i+1)$ becomes maximal when its terms are equal as possible. So, we may assume that each $a_i$ satisfies $a_i\geqslant \lfloor t/m\rfloor$, which implies that $a_i+1> t/m$, for all $i=1,\ldots,m$. In consequence, the maximum value of $p_B$ is at least $(t/m)^k$. Hence, the maximal size of the family $R'$ is at least $\binom{m}{k}(t/m)^k$, which completes the proof of the first part of the theorem.
		
		The second part follows immediately by considering $m$ and $d$ tending to infinity so that $\binom{m}{k}/d$ tends to zero. \hfill $\Box$
		
		\section{Constructions and other results}
		
		\subsection{Exemplary constructions}
		
		Let us illustrate the construction from the proof of Theorem \ref{Theorem Lower Bound} by a concrete example. Our building blocks are $1$-neighborly families of strings of the following basic form: 
		\begin{equation}
			C_1=\{0,1\}, C_2=\{00,01,1*\}, C_3=\{000,001,01*,1**\},\ldots . 
		\end{equation}
		In general, $C_d$ arises from $C_{d-1}$ by appending $0$ at the beginning of every existing string and adding $1**\cdots*$ to the collection. Clearly, $C_d$ has size $d+1$ and is a maximal $1$-neighborly family in $S^d$.
		
		Let us compute the lower bound for $n(2,7)$. By applying Proposition \ref{Proposition Lower Bound} (or directly the lower bound from (\ref{Eq Alon})), we find that $n(2,7)\geqslant n(1,3)\cdot n(1,4)=4\cdot 5=20$. An explicit family of strings of that size is $C_3C_4$. 
		
		 However, we may get a better bound by the method from the proof of Theorem \ref{Theorem Lower Bound}. Indeed, taking $m=3$, we have $7-\binom{3}{2}+1=5$, which leads to $a_1=2$, $a_2=2$, and $a_3=1$. The corresponding families of strings may be written as
		\begin{equation}
			(00)(C_2C_2*), (01)(C_2**C_1), (1*)(**C_2C_1),
		\end{equation}
		while particular strings are obtained by substituting each symbol $C_i$ with an arbitrary string from $C_i$. Thus, the sizes of these families are, respectively, $3\cdot3=9$, $3\cdot2=6$, and $3\cdot 2=6$, which makes the total size of their union equal to $9+6+6=21>20$.
		
		Though this ``fragmented'' construction gives better results for $k=2$ than a more direct one form Proposition \ref{Proposition Lower Bound}, for bigger $k$ an opposite situation may happen. Consider, for instance, the case of $n(3,10)$. Using the former method, an optimal value for the lower bound is obtained by taking $m=4$ and $a_1=a_2=a_3=2$, $a_4=1$, which gives four families of strings
		\begin{equation}
			(000)(C_2C_2C_2*), (001)(C_2C_2**C_1), (01*)C_2**C_2C_1, (1**)(**C_2C_2C_1),
		\end{equation}
	whose union has total size $3\cdot3\cdot3+(3\cdot3\cdot2)\cdot3=27+54=81$. However, using Proposition \ref{Proposition Lower Bound} and the just computed lower bound $n(2,7)\geqslant 21$, we get $n(3,10)\geqslant n(2,7)\cdot n(1,3)\geqslant 21\cdot 4=84>81$. This leads to the following definition.
	
	Let $m(k,d)$ denote the maximum size of a $k$-neighborly family in $S^d$ that can be constructed as in the proof of Theorem \ref{Theorem Lower Bound}. More explicitly, $m(k,d)$ is the maximum value of the expression $\sum_{B\in \binom{[m]}{k}}\prod_{i\in B}(a_i+1)$ over all $m$ and $a_1,\ldots,a_m$ such that $\binom{m}{k}+m-1\leqslant d$ and $a_1+\cdots+a_m\leqslant d-\binom{m}{k}-1$. For instance, $m(3,10)=81$. Clearly, we have $m(k,d)\leqslant n(k,d)$.
	
	However, to take the advantage of construction from Proposition \ref{Proposition Lower Bound} that we have seen, we define $\overline{m}(k,d)=\max\{\prod_{i=1}^{s}m(k_i,d_i)\}$, where the maximum is taken over all pairs of partitions, $k=k_1+\cdots+k_s$ and $d=d_1+\cdots+d_s$, with $1\leqslant k_i\leqslant d_i$. For instance, we have $\overline{m}(3,10)=84$. Clearly, we always have $$m(k,d)\leqslant \overline{m}(k,d)\leqslant n(k,d).$$ We shall discuss some properties of these functions in the final section.

		\subsection{Partitions and laminations}
		In this subsection we present further observations on $k$-neighborly families of boxes with additional structural properties. Staying in string terminology, we shall somewhat rely on geometric intuitions.
		
		For a given family $F \subseteq S^d$, let us define the quantity $\vol(F)=\sum_{x\in F}2^{j(x)}$, where $j(x)$ is the number of jokers in $x$. If $F$ is a $k$-neighborly family, then $\vol(F)$ is just the total number of elements of $H^d$ contained in the family of subcubes $H(x)$, with $x\in F$. Clearly, we have then $\vol(F)\leqslant 2^d$. Also notice that $2^{j(x)}$ is a real geometric volume of the standard box in $\mathscr{S}^d$ corresponding to $x$ and $\vol(F)$ is the total sum of volumes of all standard boxes corresponding to members in the family $F$.
		
		We are interested below in those $k$-neighborly families $F$ having the maximum possible volume, namely $\vol(F)=2^d$. They are simply called \emph{partitions}. We start with the following simple lemma. 
		\begin{lemma}\label{Lemma Jokers}
			Let $F\subseteq S^d$ be a $k$-neighborly family of maximum size. Then each string in $F$ may have at most $d-k$ jokers.
		\end{lemma}
		\begin{proof}
			Suppose that there is a string $x\in F$ with $d-k+1$ jokers. Let $x^{(0)}$ and $x^{(1)}$ be two strings obtained from $x$ by changing one chosen joker into $0$ and $1$, respectively. Clearly, $d(x^{(0)},x^{(1)})=1$, and for every $y\in F$, $y\neq x$, we still have $1\leqslant d(x^{(i)},y)\leqslant k$, for both $i=0,1$. Hence, $F\setminus \{x\} \cup \{x^{(0)},x^{(1)}\}$ is a strictly larger $k$-neighborly family, contrary to our assumption.
		\end{proof}
		
		\begin{proposition}
			Every $(d-1)$-neighborly family $F\subseteq S^d$ of maximum size is a partition.
		\end{proposition}
		\begin{proof}
			Let $F_0$ and $F^*$ be as in the proof of Proposition \ref{Proposition n(d-1,d)}. From this proof we know that $|F_0|\leqslant 2^{d-1}$ and $|F_0|+|F^*|=2^{d-1}+2^{d-2}$. By Lemma \ref{Lemma Jokers} we know that each string in $F^*$ has exactly one joker. It follows that $|F_0|=2^{d-1}$ and $|F^*|=2^{d-2}$, and consequently, $\vol(F)=2^d$.
		\end{proof}
		
		Our next results will have consequences for the opposite side of the scene, namely, for at most $2$-neighborly families $F$. In order to state it we need to introduce some notation and terminology.
		
		For a given family of strings $F\subseteq S^d$, denote by $F^{i,s}$ a subfamily having symbol $s\in S$ on the $i$-th position. A family $F\subseteq S^d$ is called a \emph{lamination} if it is a partition ($\vol(F)=2^d$) and $F=F^{i,0}\cup F^{i,1}$, for some $i\in [d]$. For $x\in S^d$, let ${\rm prop}(x)=\{i\in [d]\colon x_i\in \{0,1\}\}$ be the set of non-joker positions in $x$. Also, let $\sgn (x)=(-1)^{|\{i\in[d]\colon x_i=1\}|}$.
		
		The following lemma is a special case of a more general result given in \cite{KP} (Corollary 6.2).
		
		\begin{lemma}
			\label{Lemat sgn}
			Let $F\subseteq S^d$ be a partition and let $v\in F$ be a string with the least number of jokers. Then, denoting $P_v=\{x \in F\colon {\rm prop}(x)={\rm prop}(v)\}$, we have
			\begin{equation}
				\label{eq1}
				\sum_{x\in P_v}{\rm sgn}(x)=0.
			\end{equation}
		\end{lemma}
		\begin{proof}
		If ${\rm prop}(v)=[d]$, then the set of strings $F\setminus P_v$ can be partitioned into edges of $\{0,1\}^d$ (strings with one joker). Thus, a half of strings in $P_v$ contains an odd number of $1$'s, and the other half an even number of $1$'s. 
		
		Let now $|{\rm prop}(v)|=k<d$. We may assume that $v=v_1\cdots v_k*\cdots*$, where $v_i\in \{0,1\}$ for $i\in [k]$. Note that the set of strings $$G=\{w\in F: w_i\in \{0,*\}\;\; {\rm for}\;\; i\geqslant k+1\}$$ is a partition in dimension $k$. Moreover, the set $P'_v=\{w_1\cdots w_k\colon w\in P_v\}$ consists of all strings in $G$ which do not contain jokers. Thus, by the first part of the proof, the equality (\ref{eq1}) holds.
	\end{proof}
		
		A pair of strings $x,y\in S^d$  is a \emph{twin} pair if $x_i+y_i=1$ for precisely one $i\in [d]$ and $x_j=y_j$ for all $j\in [d]\setminus \{i\}$. A \emph{union} $x\cup y$ of a twin pair $x,y$ is the string $z\in S^d$ such that $z_i=*$ if $x_i+y_i=1$ and $z_j=x_j$ for all $j\in [d]\setminus \{i\}$. In our geometric interpretation, a box corresponding to $z$ is the union of boxes corresponding to $x$ and $y$.
		
		Using the above lemma we prove that among $k$-neighborly families, with $k\leqslant 2$, every partition is a lamination.
		\begin{proposition}
			\label{Proposition Lamination}
			Let $F\subseteq S^d$ be a $k$-neighborly family, with $k\leqslant 2$. If $F$ is a partition, then $F$ is a lamination.
		\end{proposition}
		
		\begin{proof}
		Let $x\in F$ be a string with the least number of jokers. Since $k\leqslant 2$, by Lemma \ref{Lemat sgn} there is $y\in F \setminus \{x\}$ such that $x,y$ is a twin pair. Let $G=(F\setminus \{x,y\})\cup \{x\cup y\}$. Obviously, the partition $G$ is still at most 2-neighborly and $|G|=|F|-1$. Continuing this process we obtain the partition $P=\{*\cdots*0*\cdots*,*\cdots*1*\cdots*\}$, where $0,1$ stand at some $i$th position. Then $F=F^{i,0}\cup F^{i,1}$.
		\end{proof}
		
		Let $v\in S^d$ be a string of length $d$. For $i \in [d]$, let $v_{-i}=v_1\cdots v_{i-1}v_{i+1}\cdots v_d$ denote the result of deletion of the $i$-th letter form $v$. More generally, for $F\subseteq S^d$ we write $F_{-i}=\{v_{-i}: v\in F \}$. A \emph{total lamination} is defined recursively as follows. For every $d\geqslant 1$, the $1$-element family consisting of a string with only jokers, as well as the full Hamming cube $H^d$, are total laminations. Next suppose that total laminations of dimensions up to $d-1$ have been already defined. Then a lamination $F\subseteq S^d$ is \emph{total} if there is $i\in[d]$ such that $F=F^{i,0}\cup F^{i,1}$ and both families, $F^{i,0}_{-i}$ and $F^{i,1}_{-i}$ are total laminations, too (see Fig. \ref{Figure Total Lamination}).
		
		It is clear that the proof of Proposition \ref{Proposition Lamination} gives actually the following stonger statement.
		\begin{corollary}
			Every at most $2$-neighborly partition in $S^d$ is a total lamination. 
		\end{corollary}
		
		A model example of a total lamination is a canonical $1$-neighborly family $C_d$. Indeed, the family $C_d$ splits along the first coordinate into $C_{d-1}$ and $\{**\cdots *\}$, which are both total laminations. For example, $C_3=\{000,001,01*,1**\}$ splits into $C_2=\{00,01,1\ast\}$ and $\{**\}$. This implies that all neighborly families obtained by the method based on Proposition \ref{Proposition Lower Bound} and Theorem \ref{Theorem Lower Bound} are total laminations, too.
		
		\begin{figure}[ht]
			
			\begin{center}
				
				\resizebox{8cm}{!}{
					
					\includegraphics{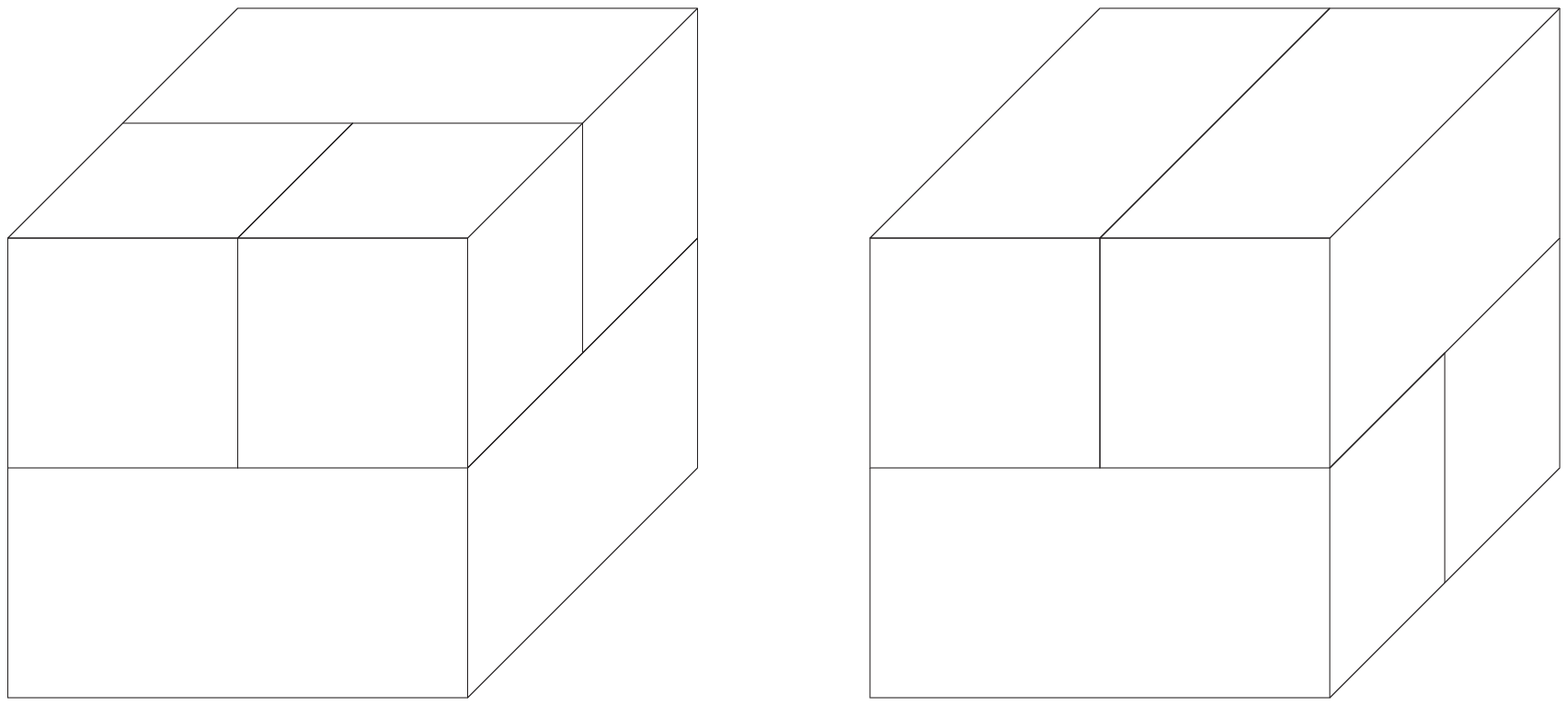}
					
				}
				
				\caption{Total laminations in dimension 3, $F=\{001,101,\ast11,\ast\ast0\}$ and $G=\{0\ast1,1\ast1,\ast00,\ast10\}$, with $|F|=|G|=n(1,3)$.}
				\label{Figure Total Lamination}
			\end{center}
			
		\end{figure}
		
		\section{Final comments}
		Let us conclude the paper with some open problems and suggestions for future studies of the function $n(k,d)$.
		
		\begin{conjecture}\label{Conjecture Total Lamination}
			For every $1\leqslant k\leqslant d$, there exists a $k$-neighborly total lamination $F$ in $S^d$ of size $n(k,d)$.
		\end{conjecture}
		
	As mentioned above, the conjecture is trivially true for $k=1$ and $k=d$. By the proof of Proposition \ref{Proposition n(d-1,d)} it is also true for $k=d-1$. The statement holds also in all cases for which the exact value of $n(k,d)$ is known, except for $n(6,8)=150$, where the only construction known to us is not a total lamination. In Table \ref{table:l(2,d)} we collected initial lower bounds for $l(2,d)$ obtained by computer experiments that improve upon the values of $m(2,d)$ (defined in Subsection 3.1), for some initial dimensions $d$.
	
		\begin{table}[h!]
		\centering
		\begin{tabular}{|c||l|l|l|l|l|l|l|l|l|l|l|l|l|l|l|l|}
			\hline
			$d$ &3&4&5& 6&7& 8& 9   & 10 & 11 & 12 & 13 & 14& 15 & 16 & 17 & 18\\
			\hline
			$m(2,d)=$&6&9&12& 16&21& 27& 33 & 40 & 48 & 56 & 65 &75& 85 & 96& 108 & 120\\
			\hline
			$l(2,d)\geqslant$&6&9&12& 16&21& 27& 33 & 40 & 48 & 57 & 67 &78& 90 & 102 & 115 & 129\\
			\hline
		\end{tabular}
		\caption{The lower bounds for $l(2,d)\leqslant n(2,d)$ obtained by constructions of total laminations compared to the function $m(k,d)\leqslant n(2,d)$ arising from Theorem \ref{Theorem Lower Bound}.}
		\label{table:l(2,d)}
	\end{table}
	
	Total laminations have many nice structural properties. For instance, one may produce out of any $k$-neighborly total lamination $F$ of dimension $d$ a sequence of similar families with all possible sizes from $|F|$ down to $1$. For example, to make a smaller family from $C_3$, take a twin pair $(000,001)$ and substitute it by their union $000\cup001=00*$. This gives a new family $\{00*,01*,1**\}$, which is still a $1$-neighborly total lamination. Repeating this step for the last family with a twin pair $(00*,01*)$ gives $\{0**,1**\}$, which again is a $1$-neighborly total lamination. In the same way we get $\{***\}$ in the final step. Thus, if Conjecture \ref{Conjecture Total Lamination} is true, then for every $(k,d)$ and every $1\leqslant p\leqslant n(k,d)$ there is a $k$-neighborly total lamination of size $p$.
		
		Though a general behavior of the function $n(k,d)$ is quite mysterious, even for fixed $k$, we propose the following conjecture stating that the numbers $n(k,d)$ satisfy a kind of ``Pascal Triangle'' property.
		
		\begin{conjecture}\label{Conjecture Pascal Triangle}
			For every $2\leqslant k\leqslant d$,
			\begin{equation}
					n(k,d)\leqslant n(k-1,d-1)+n(k,d-1).
			\end{equation}
		\end{conjecture}
	The conjecture is supported by the results of some computational experiments performed on Gurobi solver, which are included in Table \ref{table:Gurobi}. One may also consult a table with bounds for $n(k,d)$ in a recent paper \cite{ChWX}. A more theoretical argument in favor of the conjecture stems from the following proposition (whose simple proof is left to the reader).
	
	\begin{proposition}
	For every $2\leqslant k\leqslant d$,
	$$
	\overline{m}(k,d)\leqslant \overline{m}(k-1,d-1)+\overline{m}(k,d-1).
	$$
	\end{proposition}
		
		\begin{table}[h!]
			\centering
			\begin{tabular}{|c||ll|lll|llll|lllll|l|l|}
				\hline
				$d$ &5&5&6& 6&6&7&7&7&7&8&8&8& 8&8\\
				\hline
				$k$ &2&3&2& 3&4&2&3&4&5&2&3&4&5&6\\
				\hline
				Alon &12&18&16&27&36&20&36&54&72&25&48&81&108&144\\
				\hline
				Gurobi&12&\textbf{18}&16&\textbf{27}&\textbf{37}&21*&37*&54&\textbf{74}&27**&48*&81*&114& 150\\
				\hline
			\end{tabular}
			\caption{Lower bounds for $n(k,d)$ computed by Gurobi solver (compared to the lower bounds implied by (\ref{Eq Alon})). The numbers marked with an asterisk may be suboptimal solutions to the corresponding MIP problems, as computations have not been successfully completed. The number marked with a double asterisk has been determined separately. Bold numbers are exact values of $n(k,d)$, as proved in \cite{ChWX}.} 
		\label{table:Gurobi}
		\end{table}
	
Let us remark that Conjecture \ref{Conjecture Pascal Triangle} has some consequences for the presumed asymptotic properties of $n(k,d)$. In general, it seems plausible that the following conjecture is true.

\begin{conjecture}
	For every fixed $k$, there exists a real number $g_k$ such that
	\begin{equation}
	\lim_{d\rightarrow \infty}\frac{n(k,d)}{d^k}=g_k.
	\end{equation}
\end{conjecture}
	We know that $g_1=1$ by the Graham-Pollak Theorem. For $k\geqslant 2$ even the existence of the limit is not known. However, from our Theorem \ref{Theorem Lower Bound} it follows that $g_k\geqslant \frac{1}{k!}$ (assuming that $g_k$ exists). In particular, $g_2\geqslant \frac{1}{2}$. Now, using the hypothetical Pascal-Triangle inequality, we get that $g_2=\frac{1}{2}$. Indeed, by this inequality we have $n(2,j)-n(2,j-1)\leqslant j$, for any $j\geqslant 3$. Summing up these inequalities for $j=3,4,\ldots d$, we get that $n(2,d)\leqslant 4+(3+4+\cdots+d)=(d+1)+\binom{d}{2}$. So, ${n(2,d)}/d^2\leqslant \frac{1}{2}+o(1)$.

	\end{document}